\documentclass[12pt]{amsart}
\usepackage{geometry}
\usepackage{amsmath}
\usepackage{amssymb}
\usepackage{latexsym}
\usepackage{mathrsfs}
\usepackage{enumitem}
\usepackage[all]{xy}

\allowdisplaybreaks
\theoremstyle{lemma}
\newtheorem{theorem}{Theorem}[section]
\newtheorem{lemma}[theorem]{Lemma}

\theoremstyle{definition}
\newtheorem{definition}[theorem]{Definition}

\newtheorem{proposition}[theorem]{Proposition}

\theoremstyle{corollary}
\newtheorem{corollary}[theorem]{Corollary}

\theoremstyle{remark}
\newtheorem{remark}[theorem]{Remark}

\theoremstyle{notation}

\theoremstyle{claim}

\numberwithin{equation}{section}

\newtheorem{innercustomthm}{Theorem}
\newenvironment{customthm}[1]
  {\renewcommand\theinnercustomthm{#1}\innercustomthm}
  {\endinnercustomthm}

\newtheorem{innercustomcla}{Claim}
\newenvironment{customcla}[1]
  {\renewcommand\theinnercustomcla{#1}\innercustomcla}
  {\endinnercustomcla}

\begin{document}

\title[A generalized type semigroup and dynamical comparison ]{A generalized type semigroup and dynamical comparison}

\author{Xin Ma}
\email{xma29@buffalo.edu}
\address{Department of Mathematics,
         Texas A\&M University,
         Collage Station, TX 77843}
\curraddr{Department of Mathematics,
State University of New York at Buffalo,
Buffalo, NY, 14260}


\date{Feb 9, 2020.}


\begin{abstract}
In this paper, we construct and study a semigroup associated to an action of a countable discrete group on a compact Hausdorff space, that can be regarded as a higher dimensional generalization of the type semigroup. We study when this semigroup is almost unperforated. This leads to a new characterization of dynamical comparison and thus answers a question of Kerr and Schafhauser. In addition, this paper suggests a definition of comparison for dynamical systems in which neither necessarily the acting group is amenable nor the action is minimal.
\end{abstract}

\maketitle

\section{Introduction}
Reduced crossed products of the form $C(X)\rtimes_r G$ arising from the topological dynamical systems, say from $(X, G,
\alpha)$ for a countable discrete group $G$, an infinite compact Hausdorff space $X$ and a continuous action
$\alpha$, have long been an important source of examples and motivation for the study of $C^\ast$-algebras. Recently, comparison phenomena in dynamics have been observed to be essential for establishing certain structure theorems for reduced crossed product $C^*$-algebras. For example, for a minimal free action of an amenable group, it has been demonstrated in \cite{D}, \cite{D-G} and \cite{M1} that dynamical comparison implies the $\mathcal{Z}$-stability of the reduced crossed product under some assumptions by using the notion called almost finiteness defined by Kerr in \cite{D}. When the group is not amenable it is possible for action not to have an invariant probability measure, in which case dynamical comparison is verified in \cite{M2} to imply pure infiniteness of the reduced crossed products of minimal topologically free actions.

A well-known dynamical analogue of strict comparison in $C^\ast$-setting, the idea of dynamical comparison on a general compact Hausdorff space dates back to Winter in 2012 and was discussed in \cite{B} and \cite{D}. We record here the version that appeared in \cite{D}. Here $M_G(X)$ denotes the convex set of $G$-invariant regular Borel probability measures on $X$.

\begin{definition}(\cite[Definition 3.1]{D})
Let $F$ be a closed set in $X$ and $O$ a non-empty open subset of $X$. We write $F\prec O$ if there exists a finite collection $\mathcal{U}$ of open subsets of $X$ which cover $F$, an $s_U\in \Gamma$ for each $U\in \mathcal{U}$ such that the images $s_UU$ for $U\in \mathcal{U}$ are pairwise disjoint subsets of $O$. In addition, for open sets $U,V$, we write $U\prec V$ if $F\prec V$ holds whenever $F$ is a closed subset of $U$.
\end{definition}

\begin{definition}(\cite[Definition 3.2]{D})
The action $\alpha: G\curvearrowright X$ is said to have \emph{dynamical comparison} (\textit{comparison} for short) if $V\prec O$ for all open sets $V, O\subset X$ satisfying $\mu(V)<\mu(O)$ for every $\mu\in M_G(X)$.
\end{definition}

A natural question is to determine when an action has comparison.  Before the formal definition of comparison, it was well-known that all minimal
$\mathbb{Z}$-actions on the Cantor set have this property as a consequence of the Kakutani-Rokhlin clopen tower
partition (see \cite{G-W}). More recently, Downarowicz and Zhang \cite{Dow-Z} showed that all continuous actions on the Cantor set of groups for which
every finitely generated subgroup has subexponential growth have comparison. On the other hand,
it is still open whether all continuous actions on the Cantor set of amenable countably infinite groups have
comparison. However, by combining Theorem A in \cite{D-G} and Theorem 4.2 in \cite{C-J-K-M-S-T}, the property of comparison is generic for minimal free actions of a fixed amenable countably
infinite group on the Cantor set. In the setting of non-amenable groups, when there is no invariant measure for the
action, the strong boundary actions introduced in \cite{L-S} and $n$-filling actions introduced in \cite{J-R} are
natural examples of dynamical comparison. For more details see \cite{M2}.

However, observe that Definition 1.2 behaves well only when $G$ is amenable or $\alpha$ is minimal. One motivation of this paper is to find a generalized version of dynamical comparison regardless of the amenability of groups or the minimality of actions. We write the definition below. Theorem B and Theorem C in this paper will validate this generalization.

Recall that a \emph{premeasure} $\mu$ on an algebra $\mathcal{A}$ of sets is a function  $\mu: \mathcal{A}\rightarrow [0,\infty]$ satisfying the following (see \cite[~p. 30]{Folland})

\begin{enumerate}[label=(\roman*)]
\item $\mu(\emptyset)=0$;

\item $\mu(\bigsqcup_{n=1}^\infty A_n)=\sum_{n=1}^\infty \mu(A_n)$ for any disjoint sequence $\{A_n\in \mathcal{A}: n\in \mathbb{N}^+\}$ such that $\bigsqcup_{n=1}^\infty A_n\in \mathcal{A}$.
\end{enumerate}

Note that a classical theorem of Carath\'{e}odory states that each premeasure $\mu$ on an algebra $\mathcal{A}$ extends to a measure on the $\sigma$-algebra generated by $\mathcal{A}$ (see \cite[Theorem 1.14]{Folland}). In addition, if $\mu$ is $\sigma$-finite then the extension is unique.

Throughout the paper $\mathcal{A}_0$ denotes the algebra generated by open sets in $X$. We denote by $\operatorname{Pr}_G(X)$ the set of all $G$-invariant premeasures on $\mathcal{A}_0$ which are regular, i.e., having inner regularity $\mu(B)=\sup\{\mu(F): F\subset B, F\
\textrm{compact}\}$ and outer regularity $\mu(B)=\inf\{\mu(O): B\subset O, O\ \textrm{open}\}$ for all $B\in \mathcal{A}_0$. We say a premeasure $\mu\in \operatorname{Pr}_G(X)$ a \emph{probability premeasure} if $\mu(X)=1$. We remark that the extension of a premeasure $\mu\in \operatorname{Pr}_G(X)$ in the sense of Carath\'{e}odory is still $G$-invariant. If $\mu\in \operatorname{Pr}_G(X)$ is $\sigma$-finite then the unique extension is regular as well.

\begin{definition}
The action $\alpha: G\curvearrowright X$ is said to have \emph{\emph{(}generalized\emph{)} dynamical comparison} if  whenever $A, B$ are open sets in $X$ satisfying
\begin{enumerate}[label=(\roman*)]
\item $A\subset G\cdot B$;

\item $\mu(A)<1$ for every $\mu\in \operatorname{Pr}_G(X)$ with $\mu(B)=1$.
\end{enumerate}
Then $A\prec B$ holds

\end{definition}

We remark that Theorem B and Theorem C below show that if the action $\alpha$ is minimal, our generalized version of dynamical comparison (Definition 1.3) and the original one (Definition 1.2) coincide. 

In the $C^\ast$-setting, it has been proved by R{\o}rdam in \cite{R} that a simple unital $C^\ast$-algebra $A$ has strict comparison if and only if its Cuntz semigroup $\operatorname{Cu}(A)$ is almost unperforated, i.e.,
$(n+1)\cdot\langle a\rangle\leq n\cdot\langle b\rangle $ for some $n\in \mathbb{N}$ implies $\langle a\rangle\leq \langle b\rangle$. Therefore, as a dynamical analogue of strict comparison, dynamical comparison is expected to
have a characterization of the same type, without using invariant probability measures. The author was communicated by David Kerr this question, which was raised by David Kerr and Christopher Schafhauser. In this paper, we address this question and obtain the following main result (Theorem C) as a new characterization of dynamical comparison, which has the flavor of almost unperforation by Remark 1.5 below. To accomplish this goal, we consider the following order motivated by the type semigroup of zero-dimensional spaces (for example, see \cite{R-S}). Using this order we will construct a partially ordered semigroup $W(X, G)$ in the next section, called the \textit{generalized type semigroup}, which is the main tool in this paper.

\begin{definition}
Suppose that $\alpha: G\curvearrowright X$ is a continuous action of a countably infinite discrete group $G$ on a compact Hausdorff space $X$.
Let $O_1,\dots, O_n$ and $V_1,\dots, V_m$ be two sequences of open sets in $X$, We write \[\bigsqcup_{i=1}^n O_i\times \{i\}\prec \bigsqcup_{l=1}^m V_l\times \{l\}\]
if for every $i\in\{1,\dots, n\}$ and every closed set $F_i\subset O_i$ there are a collection of open sets, $\mathcal{U}_i=\{U^{(i)}_1,\dots, U^{(i)}_{J_i}\}$ forming a cover of $F_i$, $s^{(i)}_1,\dots, s^{(i)}_{J_i}\in G$ and  $k_1^{(i)},\dots, k_{J_i}^{(i)}\in \{1,\dots, m\}$ such that
\[\bigsqcup_{i=1}^n\bigsqcup_{j=1}^{J_i}s^{(i)}_jU^{(i)}_j\times \{k^{(i)}_j\}\subset \bigsqcup_{l=1}^m V_l\times \{l\}.\]
In particular, we write $nO\prec mV$ for simplification if one has
\[\bigsqcup_{i=1}^{n} O\times \{i\}\prec \bigsqcup_{l=1}^m V\times \{l\}.\]
\end{definition}

The \emph{chromatic number} of a family $\mathcal{C}$ of subsets of a given set is defined to be the least $d\in \mathbb{N}$ such that there is a partition of $\mathcal{C}$ into $d$ subcollections each of which is disjoint.

\begin{remark}
We remark that the relation $(n+1)O\prec nV$ can be described within $X$. Indeed, $(n+1)O\prec nV$ holds if and only if for every closed subset $F\subset O$ there are a family of open sets $\{U^{(i)}_j: j=1,\dots, J_i, i=1,\dots, n+1\}$, and a family of group elements $\{s^{(i)}_j\in G, j=1,\dots, J_i, i=1,\dots, n+1\}$ satisfying:
\begin{enumerate}[label=(\roman*)]
\item $F\subset \bigcup_{j=1}^{J_i} U^{(i)}_j$ for $i=1,2,\dots, n+1$,

\item $s^{(i)}_jU^{(i)}_j\subset V$ for all $j=1,\dots, J_i, i=1,\dots, n+1$, and

\item $\{s^{(i)}_jU^{(i)}_j: j=1,\dots, J_i, i=1,\dots, n+1\}$ has chromatic number at most $n$.
\end{enumerate}

\end{remark}

The following theorem can be regarded as a dynamical analogue of the equivalence between the strict comparison and almost unperforation of the Cuntz semigroup proved in \cite{R}.

\begin{customthm}{A}\emph{(Corollary 3.10)}
Suppose that $\alpha: G\curvearrowright X$ is a continuous action of a countably infinite discrete group $G$ on a compact metrizable space $X$. Then the following are equivalent.
\begin{enumerate}[label=(\roman*)]
	\item The generalized type semigroup $W(X, G)$ is almost unperforated.
	
	\item For any two sequences $(A_1,\dots, A_k)$ and $(B_1,\dots, B_l)$ of open sets in $X$, if $\bigcup_{i=1}^kA_i\subset G\cdot (\bigcup_{j=1}^l B_j)$ and $\sum_{i=1}^k\mu(A_i)<\sum_{j=1}^l\mu(B_j)$ for every $\mu\in \operatorname{Pr}_G(X)$ with  $\sum_{j=1}^l\mu(B_j)=1$, then 
	\[\bigsqcup_{i=1}^k A_i\times \{i\}\prec \bigsqcup_{j=1}^l B_j\times \{j\}\] holds.
\end{enumerate}
\end{customthm}

In addition, we obtain the following result as a characterization of our (generalized) dynamical comparison with the flavor of almost unperforation.

\begin{customthm}{B}\emph{(Corollary 3.11)}
Suppose that $\alpha: G\curvearrowright X$ is a continuous action of a countably infinite discrete group $G$ on a compact metrizable space $X$. Then the following are equivalent. 
\begin{enumerate}[label=(\roman*)]
	\item For any non-empty open sets $A, B$, if there is an $n\in \mathbb{N}^+$ such that  $(n+1)A\prec nB$ then $A\prec B$.
	
	\item The action $\alpha: G\curvearrowright X$ has (generalized) dynamical comparison in the sense of Definition 1.3.
\end{enumerate}
\end{customthm}

In light of the theorem above, when $G$ is amenable or $\alpha$ is minimal, we obtain a new characterization of dynamical comparison (Definition 1.2) as our main result in this paper addressing the question of Kerr and Schafhauser.

\begin{customthm}{C}\emph{(Corollary 3.12)}
Suppose that $\alpha: G\curvearrowright X$ is a continuous action of a countably infinite discrete group $G$ on a compact metrizable space $X$. Suppose in addition that $G$ is amenable or $\alpha$ is minimal.   The following are equivalent.
\begin{enumerate}[label=(\roman*)]
 \item Whenever $A, B$ are open sets in $X$ such that $\mu(B)>0$ for all $\mu\in M_G(X)$, if there is an $n\in \mathbb{N}^+$ such that $(n+1)A\prec nB$, then $A\prec B$.

 \item $\alpha: G\curvearrowright X$ has dynamical comparison in the sense of Definition 1.2.
\end{enumerate}
\end{customthm}

We remark that in the case that $\alpha$ is minimal, we can drop the assumption that ``$\mu(B)>0$ for all $\mu\in M_G(X)$'' in (i) in the corollary above since it holds automatically for any non-empty open set $B$.

\section{A generalized type semigroup $W(X, G)$}
The study of the type semigroup dates back to Tarski, who used this algebraic tool to study paradoxical
decompositions. In the context of topological dynamics, so far many authors have studied this topic, for example, \cite{B-L}, \cite{D}, \cite{M2},
\cite{TimR}, \cite{R-S} and \cite{Wagon}. However, the type semigroup behaves well only on zero-dimensional spaces. In this section, motivated by the classical type semigroup, we construct a generalized type
semigroup on a compact Hausdorff space (not necessarily metrizable) without any restriction on the dimension of the
space. Unlike the original one, the order on the new semigroup is not algebraic any more. Throughout the paper $G$ denotes a countably infinite discrete group, $X$ denotes an infinite compact
Hausdorff space and $\alpha: G\curvearrowright X$ denotes a continuous action of $G$ on $X$. Throughout the paper the notation  $\operatorname{supp}(f)$ for a function $f\in C(X)_+$ denotes the open support of $f$, i.e., $\operatorname{supp}(f)=\{x\in X: f(x)\neq 0\}$.

\begin{definition}
Let $\alpha: G\curvearrowright X$ be a continuous action of a countably infinite discrete group $G$ on a compact Hausdorff space $X$. Let $a=(f_1,\dots, f_n)\in C(X)_+^{\oplus n}$ and $b=(g_1,\dots, g_m)\in C(X)_+^{\oplus m}$. We write $a\preccurlyeq b$ if \[\bigsqcup_{i=1}^n \operatorname{supp}(f_i)\times \{i\}\prec \bigsqcup_{l=1}^m \operatorname{supp}(g_l)\times \{l\}\] holds in the sense of Definition 1.4.
\end{definition}

We write $K(X, G)=\bigcup_{n=1}^\infty C(X)_+^{\oplus n}$ and observe that the relation $\preccurlyeq$ described above is in fact defined on $K(X, G)$. We remark that Definition 1.4 allows us to describe the subequivalence relation $\preccurlyeq$ by simply using open sets like the classical type semigroup in the context of zero-dimensional spaces. However, we insist on considering functions because the relation $\preccurlyeq$ between two sequences of functions $a, b\in K(X, G)$ is naturally related to the Cuntz subequivalence $\precsim$ for $a$ and $b$ in the $C^*$-algebra $C(X)\rtimes_r G$ (see Proposition 2.3 below). To investigate properties of the relation $\preccurlyeq$, we first show that this relation is transitive.

\begin{lemma}
Let $a,b,c\in K(X, G)$ be such that $a\preccurlyeq b$ and $b\preccurlyeq c$. Then $a\preccurlyeq c$.
\end{lemma}
\begin{proof}
First we write $a=(f_1,\dots, f_N)$, $b=(g_1,\dots, g_L)$ and $c=(h_1,\dots, h_M)$ for some integers $N,L,M\in \mathbb{N}^+$. Since $a\preccurlyeq b$, one has that for every $n\in\{1,\dots, N\}$ and closed set $F_n\subset \operatorname{supp}(f_n)$ there are a collection of open sets $\mathcal{U}_n=\{U^{(n)}_1,\dots, U^{(n)}_{J_n}\}$ forming a cover of $F_n$, $s^{(n)}_1,\dots, s^{(n)}_{J_n}\in G$ and  $k_1^{(n)},\dots, k_{J_n}^{(n)}\in \{1,\dots, L\}$ such that
\[\bigsqcup_{n=1}^N\bigsqcup_{j=1}^{J_n}s^{(n)}_jU^{(n)}_j\times \{k^{(n)}_j\}\subset \bigsqcup_{l=1}^L \operatorname{supp}(g_l)\times \{l\}.\]
Then compactness and normality of the space $X$ shows that there is a family of open sets $\{V^{(n)}_j: j=1,\dots, J_n, n=1,\dots, N\}$ such that for each $n$ the collection $\mathcal{V}_n=\{V^{(n)}_j: j=1,\dots, J_n\}$ is a cover of $F_n$ and $\overline{V^{(n)}_j}\subset U_j^{(n)}$ for every $j=1,\dots, J_n$. Therefore, one has
\[\bigsqcup_{n=1}^N\bigsqcup_{j=1}^{J_n}s^{(n)}_j\overline{V^{(n)}_j}\times \{k^{(n)}_j\}\subset \bigsqcup_{l=1}^L \operatorname{supp}(g_l)\times \{l\}.\]

Define $\mathcal{D}_{l}=\{s^{(n)}_j\overline{V^{(n)}_j}: k^{(n)}_j=l, j=1,\dots, J_n, n=1,\dots, N\}$ and write $K_l=\bigsqcup \mathcal{D}_{l}$, which is closed and a subset of $\operatorname{supp}(g_l)$. Now because $b\preccurlyeq c$, for all $K_l\subset \operatorname{supp}(g_l)$ there are a collection of open sets $\mathcal{W}_l=\{W^{(l)}_1,\dots, W^{(l)}_{P_l}\}$ forming a cover of $K_l$, $t^{(l)}_1,\dots, t^{(l)}_{P_l}\in G$ and  $d_1^{(l)},\dots, d_{P_l}^{(l)}\in \{1,\dots, M\}$ such that
\[\bigsqcup_{l=1}^L\bigsqcup_{p=1}^{P_l}t^{(l)}_pW^{(l)}_p\times \{d^{(l)}_p\}\subset \bigsqcup_{m=1}^M \operatorname{supp}(h_m)\times \{m\}.\]

Define $R_{n,j,p,l}=V^{(n)}_j\cap {(s^{(n)}_j)}^{-1}W_p^{(l)}$ for $n, j, p, l$ satisfying $k^{(n)}_j=l$. Then we observe that the family
\[\mathcal{R}_n=\{R_{n,j,p,l}:j=1,\dots, J_n, l=1,\dots, L, k^{(n)}_j=l, p=1,\dots, P_l\} \]
forms an open cover of $F_n$. Indeed, first fix an $x\in F_n$. Then there is an $V_j^{(n)}$ such that $x\in V_j^{(n)}$. Now taking $l=k^{(n)}_j$ we have $s^{(n)}_j\overline{V^{(n)}_j}\subset K_l\subset \bigcup_{p=1}^l W_p^{(l)}$, which implies that $s^{(n)}_jx\in s^{(n)}_jV^{(n)}_j\cap W_p^{(l)}$ for some $p\leq P_l$. Thus, we have $x\in R_{n,j,p,l}$.

In addition, we define $r_{n,j,p,l}=t^{(l)}_ps^{(n)}_j\in G$ for $n,j, p,l$ satisfying $k^{(n)}_j=l$. Now, we claim that the family $\mathcal{T}=\{r_{n,j,p,l}R_{n,j,p,l}\times \{d^{(l)}_p\}:j=1,\dots, J_n, l=1,\dots, L, k^{(n)}_j=l, p=1,\dots, P_l\}$ is disjoint. To simplify the notation, we write $T_{n,j,p,l}=r_{n,j,p,l}R_{n,j,p,l}\times \{d^{(l)}_p\}$ and have
\[T_{n,j,p,l}=(t^{(l)}_ps^{(n)}_jV^{(n)}_j\cap t^{(l)}_pW_p^{(l)})\times \{d^{(l)}_p\}\subset t^{(l)}_pW^{(l)}_p\times \{d^{(l)}_p\}.\]
Now, suppose that $T_{n_1,j_1,p_1,l_1}$ and $T_{n_2,j_2,p_2,l_2}$ are different. If $l_1\neq l_2$ or $p_1\neq p_2$ then by our construction one has
\[(t^{(l_1)}_{p_1}W^{(l_1)}_{p_1}\times \{d^{(l_1)}_{p_1}\})\cap (t^{(l_2)}_{p_2}W^{(l_2)}_{p_2}\times \{d^{(l_2)}_{p_2}\})=\emptyset,\]
which implies that $T_{n_1,j_1,p_1,l_1}\cap T_{n_2,j_2,p_2,l_2}=\emptyset$. Otherwise we have $n_1\neq n_2$ or $j_1\neq j_2$ while there are $l$ and $p$ such that $l_1=l_2=l$, $p_1=p_2=p$ and $k^{(n_1)}_{j_1}=k^{(n_2)}_{j_2}=l$. In this case, first by the construction one has
\[(s^{(n_1)}_{j_1}V^{(n_1)}_{j_1}\times \{k^{(n_1)}_{j_1}\})\cap (s^{(n_2)}_{j_2}V^{(n_2)}_{j_2}\times \{k^{(n_2)}_{j_2}\})=\emptyset.\]
Thus, $s^{(n_1)}_{j_1}V^{(n_1)}_{j_1}\cap s^{(n_2)}_{j_2}V^{(n_2)}_{j_2}=\emptyset$ because $k^{(n_1)}_{j_1}=k^{(n_2)}_{j_2}=l$. This fact shows $T_{n_1,j_1,p_1,l_1}\cap T_{n_2,j_2,p_2,l_2}=\emptyset$ as desired. So far we have verified that the family $\mathcal{T}$ above is disjoint.

On the other hand, considering the fact that
\[T_{n,j,p,l}\subset t^{(l)}_pW^{(l)}_p\times \{d^{(l)}_p\} \subset \bigsqcup_{m=1}^M\operatorname{supp}(h_m)\times \{m\}\]
for all $T_{n,j,p,l}$, we have established the relation
\[\bigsqcup_{n=1}^N\bigsqcup_{l=1}^L \bigsqcup_{p=1}^{P_l}\bigsqcup_{\{1\leq j\leq J_n: k^{(n)}_j=l\}}r_{n,j,p,l}R_{n,j,p,l}\times \{d^{(l)}_p\}\subset \bigsqcup_{m=1}^M \operatorname{supp}(h_m)\times \{m\},\]
which verifies that $a\preccurlyeq c$ as desired.
\end{proof}

Now we can define a relation on $K(X, G)$ by setting $a\approx b$ if $a\preccurlyeq b$  and $b\preccurlyeq a$ for $a,b\in K(X, G)$. To see that this relation is in fact an equivalence relation, first it is not hard to verify directly that $a\approx a$ for all $a\in K(X, G)$ . In addition, by the definition of the relation $\approx$, one has $a\approx b$ implying $b\approx a$ trivially. Now suppose $a\approx b$ and $b\approx c$. By definition one has $a\preccurlyeq b\preccurlyeq c$ and $c\preccurlyeq b\preccurlyeq a$. Then Lemma 2.2 entails that $a\preccurlyeq c$ and $c\preccurlyeq a$. This establishes $a\approx c$.

We write $W(X, G)$ for the quotient $K(X, G)/\approx$ and define an operation on $W(X, G)$ by $[a]+[b]=[(a,b)]$, where $(a,b)$ is defined to be the concatenation of $a=(f_1,\dots, f_n)$ and $b=(g_1,\dots, g_m)$, i.e., $(a,b)=(f_1,\dots, f_n,g_1,\dots, g_m)$. It is not hard to see that if $a_1\preccurlyeq a_2$ and $b_1\preccurlyeq b_2$ then $(a_1, b_1)\preccurlyeq (a_2, b_2)$. Then Lemma 2.2 implies the operation is well-defined and it can be additionally verified that the operation is abelian, i.e, $[a]+[b]=[b]+[a]$.  Moreover, we endow $W(X, G)$ with the natural order by declaring $[a]\leq [b]$ if $a\preccurlyeq b$. Thus $W(X, G)$ is a well-defined abelian partially ordered semigroup.

The following proposition shows that our relation $\preccurlyeq$ naturally relates to the Cuntz subequivalence relation in the context of $C^*$-algebras. For Cuntz subequivalence relation, we refer to \cite{A-P-T} as a reference. Let $A$ be a $C^\ast$-algebra. We write $M_\infty(A)=\bigcup_{n=1}^\infty M_n(A)$ (viewing $M_n(A)$ as an upper left-hand corner in $M_m(A)$ for $m>n$). Let $a,b$ be two positive elements in $M_n(A)_+$ and $M_m(A)_+$, respectively. We write $a\precsim b$ if there exists a sequence $(r_n)$ in $M_{m,n}(A)$ with $r_n^\ast
br_n\rightarrow a$. We denote by $\operatorname{Diag}(f_1,\dots, f_n)$ the diagonal matrix with entries $f_1,\dots, f_n$. For every $f\in C(X)_+$ and $\epsilon>0$,  the function $(f-\epsilon)_+$ is defined via the functional calculus as $H_{\epsilon}(f)$ in the $C^*$-algebra $C(X)$ where $H_{\epsilon}(t)=\max\{t-\epsilon, 0\}$. Note that $(f-\epsilon)_+(x)=f(x)-\epsilon$ if $f(x)\geq \epsilon$ while $(f-\epsilon)_+(x)=0$ if $f(x)< \epsilon$.

\begin{proposition}
Let $a=(f_1,\dots, f_n)$ and $b=(g_1,\dots, g_m)\in K(X, G)$. If $a\preccurlyeq b$ then $\operatorname{Diag}(f_1,\dots, f_n)\precsim\operatorname{Diag}(g_1,\dots, g_m)$
in the $C^\ast$-algebra $C(X)\rtimes_r G$.
\end{proposition}
\begin{proof}
In light of Proposition 2.17 in \cite{A-P-T}, it suffices to prove that $\operatorname{Diag}((f_1-\epsilon)_+,\dots, (f_n-\epsilon)_+)\precsim\operatorname{Diag}(g_1,\dots, g_m)$ for all $\epsilon>0$. Now, let $\epsilon>0$ and define $F_i=\overline{\operatorname{supp}((f_i-\epsilon)_+)}$ for $i=1,\dots, n$. Since $a\preccurlyeq b$, for $i=1,\dots, n$ there are a collection of open sets, $\mathcal{U}_i=\{U^{(n)}_i,\dots, U^{(i)}_{J_i}\}$ forming a cover of $F_i$, $s^{(i)}_1,\dots, s^{(i)}_{J_i}\in G$ and  $k_1^{(i)},\dots, k_{J_i}^{(i)}\in \{1,\dots, m\}$ such that
\[\bigsqcup_{i=1}^n\bigsqcup_{j=1}^{J_i}s^{(i)}_jU^{(i)}_j\times \{k^{(i)}_j\}\subset \bigsqcup_{l=1}^m \operatorname{supp}(g_l)\times \{l\}.\]

Let $\{h_j^{i}: j=1,\dots, J_i\}$ be a partition of unity subordinate to the cover $\mathcal{U}_i$ of $F_i$.
Then $F_i\subset \operatorname{supp}(\sum_{j=1}^{J_i}h_j^i)$, which implies that $(f_i-\epsilon)_+\precsim \sum_{j=1}^{J_i}h_j^i$ by Proposition 2.5 in \cite{A-P-T}. Then we have
\[\operatorname{Diag}((f_1-\epsilon)_+,\dots, (f_n-\epsilon)_+)=\bigoplus_{i=1}^n(f_i-\epsilon)_+\precsim\bigoplus_{i=1}^n(\sum_{j=1}^{J_i}h_j^i)\precsim \bigoplus_{i=1}^n\bigoplus_{j=1}^{J_i}h_j^i.\]
Define a unitary $u=\bigoplus_{i=1}^n\bigoplus_{j=1}^{J_i}u_{s^{(i)}_j}$, where all $u_{s^{(i)}_j}$ are canonical unitaries in the crossed product. Then we have
\[\bigoplus_{i=1}^n\bigoplus_{j=1}^{J_i}h_j^i\sim u(\bigoplus_{i=1}^n\bigoplus_{j=1}^{J_i}h_j^i)u^*=\bigoplus_{i=1}^n\bigoplus_{j=1}^{J_i}\alpha_{s^{(i)}_j}(h_j^i).\]
To simplify the notation, we define the index set $\mathcal{I}_l=\{(i,j): j=1,\dots, J_i, i=1,\dots, n,  k_{j}^{(i)}=l\}$. Then observe that the collection $\{\operatorname{supp}(\alpha_{s^{(i)}_j}(h_j^i))\subset s^{(i)}_jU_j^{(i)}: (i,j)\in \mathcal{I}_l\}$ is disjoint for each $l=1,\dots, m$. This implies that
\[\bigoplus_{i=1}^n\bigoplus_{j=1}^{J_i}\alpha_{s^{(i)}_j}(h_j^i)\sim \bigoplus_{l=1}^m\bigoplus_{(i,j)\in \mathcal{I}_l}\alpha_{s^{(i)}_j}(h_j^i)\sim \bigoplus_{l=1}^m(\sum_{(i,j)\in \mathcal{I}_l}\alpha_{s^{(i)}_j}(h_j^i)).\]
Finally, note that
\[\operatorname{supp}(\sum_{(i,j)\in \mathcal{I}_l}\alpha_{s^{(i)}_j}(h_j^i))=\bigsqcup_{(i,j)\in \mathcal{I}_l}\operatorname{supp}(\alpha_{s^{(i)}_j}(h_j^i))\subset \operatorname{supp}(g_l)\] for each $l=1,\dots, m$. This implies that $\sum_{(i,j)\in \mathcal{I}_l}\alpha_{s^{(i)}_j}(h_j^i)\precsim g_l$, which further entails that
\[\bigoplus_{l=1}^m(\sum_{(i,j)\in \mathcal{I}_l}\alpha_{s^{(i)}_j}(h_j^i))\precsim \bigoplus_{l=1}^mg_l=\operatorname{Diag}(g_1,\dots, g_m).\]

We have verified that
\[\operatorname{Diag}((f_1-\epsilon)_+,\dots, (f_n-\epsilon)_+)\precsim \operatorname{Diag}(g_1,\dots, g_m)\] for every $\epsilon>0$ and thus we have $\operatorname{Diag}(f_1,\dots, f_n)\precsim\operatorname{Diag}(g_1,\dots, g_m)$.
\end{proof}

We end this section by remarking that our generalized type semigroup $W(X, G)$ can also be used to study paradoxical decompositions in the context of topological dynamics. The paradoxical decomposition property can be formulated by $2[a]\leq [a]$ in $W(X, G)$ for all $a\in K(X, G)$. Note that this condition is equivalent to a notion for an action called \textit{paradoxical comparison} introduced in \cite{M2} where paradoxical comparison is also proved implying the pure infiniteness of the reduced crossed product $C(X)\rtimes_r G$ in the case that there are finitely many $G$-invariant closed sets.

\section{States on $W(X, G)$ and the proof of Theorem A}
We first recall some general background information about states on preordered abelian semigroups.

A \textit{state} on a preordered monoid $(W, +, \leq)$ is an order-preserving morphism $D: W\rightarrow [0,\infty]$ with $D(0)=0$. We denote by $S(W)$ the set consisting of all states of $W$. We write $S(W, x)=\{D\in S(W): D(x)=1\}$. The following proposition due to Ortega, Perera, and R{\o}rdam is very useful.

\begin{proposition}(\cite[Proposition 2.1]{O-P-R})
Let $(W,+,\leq)$ be an ordered abelian semigroup, and let $x,y\in W$. Then the following conditions are equivalent:
\begin{enumerate}[label=(\roman*)]
\item There exists $k\in \mathbb{N}$ such that $(k+1)x\leq ky$.

\item There exists $k_0\in \mathbb{N}$ such that $(k+1)x\leq ky$ for every $k\geq k_0$.

\item There exists $m\in \mathbb{N}$ such that $x\leq my$ and $D(x)<D(y)$ for every state $D\in S(W,y)$.
\end{enumerate}

\end{proposition}

In this section, we always assume that the space $X$ is  metrizable. In addition, for $a=(f_1,\dots, f_n)\in K(X, G)$, we denote by $(a-\epsilon)_+$ the element $((f_1-\epsilon)_+,\dots, (f_n-\epsilon)_+)$ in $K(X, G)$. It is not hard to verify $((a-\epsilon)_+-\delta)_+ =(a-\epsilon-\delta)_+$ for $a\in K(X, G)$, $\epsilon>0$ and $\delta>0$.

In parallel with the Cuntz semigroup (for example, see \cite{A-P-T}), we have the following fact.

\begin{proposition}
For all $a, b\in K(X, G)$, the following are equivalent.
\begin{enumerate}[label=(\roman*)]
\item $a\preccurlyeq b$;

\item for all $\epsilon>0$ one has $(a-\epsilon)_+\preccurlyeq b$;

\item for all $\epsilon>0$  there exists a $\delta>0$ such that $(a-\epsilon)_+\preccurlyeq (b-\delta)_+$;
\end{enumerate}
\end{proposition}
\begin{proof}
Write $a=(f_1,\dots, f_n)$ and $b=(g_1,\dots, g_m)$. Then by definition we have $(a-\epsilon)_+=((f_1-\epsilon)_+,\dots, (f_n-\epsilon)_+)$. Consider for each $1\leq i\leq n$, one has  $\operatorname{supp}((f_i-\epsilon)_+)\subset \operatorname{supp}(f_i)$, which implies that $(a-\epsilon)_+\preccurlyeq a$. This fact shows that (i)$\Rightarrow$(ii).

To show (ii)$\Rightarrow$(i), first for every $1\leq i\leq n$ and closed set $F_i\subset
\operatorname{supp}(f_i)$, there is an $\epsilon_i>0$ such that $F_i\subset \{x\in X : f_i(x)>\epsilon_i\}\subset \operatorname{supp}(f_i)$. Define $\epsilon=\min\{\epsilon_i: 1\leq i\leq n\}$. Then, $F_i\subset \{x\in X :
f_i(x)>\epsilon\}=\operatorname{supp}((f_i-\epsilon)_+)\subset \operatorname{supp}(f_i)$ for all $i$. Now, since
$(a-\epsilon)_+\preccurlyeq b$, there are a collection of open sets $\mathcal{U}_i=\{U^{(i)}_1,\dots,
U^{(i)}_{J_i}\}$ forming a cover of $F_i$, $s^{(i)}_1,\dots, s^{(i)}_{J_i}\in G$ and  $k_1^{(i)},\dots,
k_{J_i}^{(i)}\in \{1,\dots, m\}$ such that
\[\bigsqcup_{i=1}^n\bigsqcup_{j=1}^{J_i}s^{(i)}_jU^{(i)}_j\times \{k^{(i)}_j\}\subset \bigsqcup_{l=1}^m \operatorname{supp}(g_l)\times \{l\}.\]
But this implies that $a\preccurlyeq b$.

Now, suppose that (iii) holds. Then for every $\epsilon>0$, by combining arguments in the two directions, one has $(a-\epsilon)_+\preccurlyeq (b-\delta)_+\preccurlyeq b$ and thus $a\preccurlyeq b$. This establishes
(iii)$\Rightarrow$(i). It is left to show (i)$\Rightarrow$(iii).  Indeed, by the definition of $a\preccurlyeq b$ and the compactness and normality of the space, for every $i\in\{1,\dots, n\}$ and closed set $\overline{\operatorname{supp}((f_i-\epsilon)_+)}:=F_i\subset
\operatorname{supp}(f_i)$ there are a collection of open sets $\mathcal{U}_i=\{U^{(i)}_1,\dots, U^{(i)}_{J_i}\}$ forming a cover of $F_i$, $s^{(i)}_1,\dots, s^{(i)}_{J_i}\in G$ and  $k_1^{(i)},\dots, k_{J_i}^{(i)}\in
\{1,\dots, m\}$ such that
\[\bigsqcup_{i=1}^n\bigsqcup_{j=1}^{J_i}s^{(i)}_j\overline{U^{(i)}_j}\times \{k^{(i)}_j\}\subset \bigsqcup_{l=1}^m \operatorname{supp}(g_l)\times \{l\}.\]

Define $\mathcal{D}_{l}=\{s^{(i)}_j\overline{U^{(i)}_j}: j=1,\dots, J_i, i=1,\dots, n, k^{(i)}_j=l\}$ and write $K_l=\bigsqcup \mathcal{D}_{l}$, which is a closed subset of $\operatorname{supp}(g_l)$. Then there is a $\delta_l$ such that $K_l\subset \{x\in X: g_l(x)>\delta_l\}\subset \operatorname{supp}(g_l)$. Setting $\delta=\min\{\delta_l: 1\leq l\leq m\}$ we have $(a-\epsilon)_+\preccurlyeq (b-\delta)_+$.
\end{proof}

\begin{definition}
A state $D$ on the semigroup $W(X, G)$ is called \textit{lower semi-continuous} if $D([a])=\sup_{\epsilon>0}D([(a-\epsilon)_+])$ for all $a\in K(X, G)$.
\end{definition}

For every state $D\in S(W(X, G))$, define $\overline{D}([a])=\sup_{\epsilon>0}D([(a-\epsilon)_+])$, which is always a lower semi-continuous state on $W(X, G)$.

\begin{proposition}
For each state $D\in S(W(X, G))$, the induced function $\overline{D}$ is a lower semi-continuous state.
\end{proposition}
\begin{proof}
Let $a\preccurlyeq b$. Then by the proposition above, for all $\epsilon>0$ there is a $\delta>0$ such that $(a-\epsilon)_+\preccurlyeq (b-\delta)_+$. Thus, $\overline{D}([a])=\lim_{\epsilon\to 0}D([(a-\epsilon)_+])\leq \lim_{\delta\to 0}D([(b-\delta)_+])=\overline{D}([b])$. This shows that $\overline{D}$ is monotone.

Let $a, b\in K(X, G)$. If $\overline{D}([a])$ or $\overline{D}([b])$ is infinite then $\overline{D}([a]+[b])=\overline{D}([a])+\overline{D}([b])$ holds trivially since $\overline{D}$ is monotone. We then assume that both of them are finite. Then in this case one has
\begin{align*}
\overline{D}([a]+[b])&=\lim_{\epsilon\to 0}D([((a,b)-\epsilon)_+])=\lim_{\epsilon\to 0}D([((a-\epsilon)_+, (b-\epsilon)_+)])\\
&=\lim_{\epsilon\to 0}D([(a-\epsilon)_+])+\lim_{\epsilon\to 0}D([(b-\epsilon)_+])\\
&=\overline{D}([a])+\overline{D}([b]).
\end{align*}
This verifies that $\overline{D}$ is a state.

For lower semi-continuity, note that
\[\overline{D}([(a-\epsilon)_+])=\lim_{\delta\to 0}D([((a-\epsilon)_+-\delta)_+])=\lim_{\delta\to 0}D([(a-\epsilon-\delta)_+]).\] Thus we have
\[\lim_{\epsilon\to 0}\overline{D}([(a-\epsilon)_+])=\lim_{\epsilon\to 0}\lim_{\delta\to 0}D([(a-\epsilon-\delta)_+])=\overline{D}([a]).\]
\end{proof}

For every premeasure $\mu$ in $\operatorname{Pr}_G(X)$ define a state $D_\mu$ on $W(X, G)$ by $D_\mu([a])=\sum_{i=1}^n\mu(\operatorname{supp}(f_i))$ for $a=(f_1,\dots, f_n)\in K(X, G)$.

\begin{proposition}
$D_\mu$ defined above is a lower semi-continuous state on $W(X, G)$.
\end{proposition}
\begin{proof}
The additivity of $D_\mu$ is clear from the definition of $D_\mu$ above. Since $\mu\in \operatorname{Pr}_G(X)$ is $G$-invariant and inner-regular, we see that if $a\preccurlyeq b$ then $D_\mu([a])\leq D_\mu([b])$. Now, let $a=(f_1,\dots, f_n)\in K(X, G)$. For every $1\leq i\leq n$ and a closed set $F_i\subset \operatorname{supp}(f_i)$ there is an $\epsilon_i$ such that $F_i\subset\operatorname{supp}((f_i-\epsilon_i)_+)=\{x\in X: f_i(x)>\epsilon_i\}\subset \operatorname{supp}(f_i)$. Now let $\epsilon=\max\{\epsilon_i: 1\leq i\leq n\}$ and thus $\sum_{i=1}^n \mu(F_i)\leq D_\mu([(a-\epsilon)_+])\leq D_\mu([a])$, which implies that $\sup_{\epsilon> 0}D_\mu([(a-\epsilon)_+])=D([a])$ because $\mu$ is inner-regular for every $\operatorname{supp}(f_i)$.
\end{proof}

We will show in Lemma 3.6 that the converse of  Proposition 3.5 is also true, that is, every lower semi-continuous state $D$ on $W(X, G)$ is of the form $D_\mu$ for a premeasure $\mu\in \operatorname{Pr}_G(X)$. The proof of this fact has a classical flavor. It is routine but quite long. In the Cuntz semigroup setting, Blackadar and Handelman provided a version concerning bounded dimension functions, which are bounded states of the Cuntz semigroup (see \cite[Proposition I.2.1]{B-H}). However, they omitted the proof. In addition, R{\o}rdam and Sierakowski proved the result for the type semigroup (see \cite[Lemma 5.1]{R-S}) in the zero-dimensional setting. Their proof relies on the zero-dimensionality of the space and cannot be generalized to higher dimensional cases.  Therefore, for the convenience of the readers, we present the proof here.  We denote by $\operatorname{Lsc}(W(X, G))$ the set of all lower semi-continuous states on $W(X, G)$.

\begin{lemma}
Every lower semi-continuous state $D\in \operatorname{Lsc}(W(X, G))$ induces a $G$-invariant premeasure $\mu_D\in \operatorname{Pr}_G(X)$.
\end{lemma}
\begin{proof}First for every open set $O$, define $\mu_D(O)=D([f])$ where $O=\operatorname{supp}(f)$ for some $f\in C(X)_+$. Then by the definition of state, $\mu_D$ is $G$-invariant on open sets. In addition, it is finitely subadditive on open sets, i.e., if $O_1,\dots O_n$ are open then
\[\mu_D(\bigcup_{i=1}^n O_i)\leq \sum_{i=1}^n \mu_D(O_i).\]
Moreover, if the $O_1,\dots O_n$ are pairwise disjoint then we have additivity:
\[\mu_D(\bigsqcup_{i=1}^n O_i)= \sum_{i=1}^n \mu_D(O_i).\]
Finally, $\mu_D$ is monotone for open sets, i.e., $O_1\subset O_2$ implies $\mu_D(O_1)\leq \mu_D(O_2)$.
For every closed set $F$, define $\mu_D(F)=\inf\{\mu_D(O): F\subset O, O\ \textrm{open}\}$. Since the space $X$ is normal, $\mu_D$ is additive with respect to disjoint closed sets $\{F_1,\dots, F_n\}$, i.e.,
\[\mu_D(\bigsqcup_{i=1}^n F_i)= \sum_{i=1}^n \mu_D(F_i).\]

\begin{customcla}{1}
Let $F$ be a closed set and  $\{F_n\}$ an increasing sequence such that $F=\bigcup_{n=1}^\infty F_n$. Then $\mu_D(F)=\sup_n\mu_D(F_n)$.
\end{customcla}
\begin{proof}
If one of $\mu_D(F_n)$ is infinite, then this equality holds trivially. Thus, we may assume each of $\mu_D(F_n)$ is finite.  Fix an $\epsilon>0$. By the definition of $\mu_D(F_n)$, for each $n$ there is an open set $O_n\supset F_n$ such that
\[|\mu_D(F_n)-\mu_D(O_n)|\leq \epsilon/2^n.\]
Then $F\subset \bigcup_{n=1}^\infty O_n$ and thus there is an $N>0$ such that $F\subset \bigcup_{n=1}^N O_n$. Note that
\[(\bigcup_{n=1}^NO_n)\setminus F_N\subset \bigcup_{n=1}^N(O_n\setminus F_n),\]
which implies that
\[\mu_D((\bigcup_{n=1}^NO_n)\setminus F_N)\leq \mu_D(\bigcup_{n=1}^N(O_n\setminus F_n))\leq\sum_{n=1}^N \mu_D(O_n\setminus F_n)\leq \epsilon.\]

Write $O=\bigcup_{n=1}^NO_n$ for simplicity. We have $(O\setminus F_N)\sqcup F_N=O$. Now for every open set $U\supset F_N$, one has $O\subset O\setminus F_N \cup U$, which entails that
\[\mu_D(U)\geq \mu_D(O)-\mu_D(O\setminus F_N).\]
Therefore, one has $\mu_D(F_N)\geq \mu_D(O)-\mu_D(O\setminus F_N)\geq \mu_D(O)-\epsilon$. As $F\subset O$, one has
\[\mu_D(F_N)\geq \mu_D(O)-\epsilon\geq \mu_D(F)-\epsilon,\]
which establishes the claim.
\end{proof}

Now, define $$\mu_D(A)=\sup\{\mu_D(K): K\subset A, K\ \textrm{closed}\}$$ for every $F_\sigma$ set $A$. We need to verify that this definition coincides with our original definition for open sets at the beginning of this
 proof. Indeed, let $O=\operatorname{supp}(f)$ for some continuous function $f$.  Since $D$ is lower semi-continuous, one has 
\[\mu_D(O)=D([f])=\sup_{\epsilon>0}D([(f-\epsilon)_+])= \sup_{\epsilon>0}\mu_D(\operatorname{supp}((f-\epsilon)_+)).\]
Then observe that
\[\operatorname{supp}((f-\epsilon)_+)\subset \overline{\operatorname{supp}((f-\epsilon)_+)}\subset \operatorname{supp}(f)=O,\] which implies that \[\mu_D(O)=\sup\{\mu_D(F): F\subset O, F\ \textrm{closed}\}\] as desired.  Therefore, $\mu_D$ is well-defined on $F_\sigma$ sets.  Observe that it is monotone for all $F_\sigma$ sets.

\begin{customcla}{2}
Let $A=\bigcup_{n=1}^\infty F_n$ for an increasing sequence of closed sets $\{F_n\}$. Then $\mu_D(A)=\sup_{n}\{\mu_D(F_n)\}$.
\end{customcla}
\begin{proof}
By definition it suffices to show $\mu_D(A)=\sup\{\mu_D(K): K\subset A, K\ \textrm{closed}\}\leq \sup_{n}\{\mu_D(F_n)\}$. The proof is similar to that of Claim 1. If one of $\mu_D(F_n)$ is infinite, then the equality above holds trivially. Thus, we may assume each $\mu_D(F_n)$ is finite. Fix an $\epsilon>0$ and a closed set $K\subset A$. By the definition of $\mu_D(F_n)$, for each $n$ there is an open set $O_n\supset F_n$ such that
\[|\mu_D(F_n)-\mu_D(O_n)|\leq \epsilon/2^n.\]
Then $K\subset A\subset \bigcup_{n=1}^\infty O_n$ and thus there is an $N>0$ such that $K\subset \bigcup_{n=1}^N O_n$. Then because $\{F_n\}$ is increasing, one has
\[K\setminus F_N\subset (\bigcup_{n=1}^NO_n)\setminus F_N\subset \bigcup_{n=1}^N(O_n\setminus F_n),\]
Note that $K\setminus F_N$ is also a $F_\sigma$ set. Then we have
\[\mu_D(K\setminus F_N)\leq \mu((\bigcup_{n=1}^NO_n)\setminus F_N)\leq \sum_{n=1}^N \mu_D(O_n\setminus F_n)\leq \epsilon\]
since $\mu_D$ is monotone on $F_\sigma$ sets. We write $K\setminus F_N=\bigcup_{n=1}^\infty P_n$ for an increasing sequence of closed sets $\{P_n\}$. Then $K=(K\setminus F_N)\sqcup (K\cap F_N)= \bigsqcup_{n=1}^\infty((K\cap F_N)\sqcup P_n)$. Then claim 1 entails that $\mu_D(K)=\sup_n\{\mu_D((K\cap F_N)\sqcup P_n)\}$. Now there is a $M>0$ such that
\[\mu_D((K\cap F_N))+\mu_D(P_M)=\mu_D((K\cap F_N)\sqcup P_M)\geq \mu_D(K)-\epsilon.\]
Thus, we have
\[\mu_D(F_N)\geq \mu_D((K\cap F_N))\geq \mu_D(K)-2\epsilon.\]
This establishes Claim 2.
\end{proof}

Now, consider the \textit{semialgebra} $\mathcal{S}=\{O\cap F: O\ \textrm{open}, F\ \textrm{closed}\}$ in the sense of \cite[p.~297]{Roy}. Since our $X$ is metrizable, every set $O\cap F\in \mathcal{S}$ is a $F_\sigma$ set. Recall that $\mathcal{A}_0$ denotes the algebra generated by open sets in $X$. Note that $\mathcal{A}_0$ equals $\{\bigcup_{i=1}^n C_i: C_i\in \mathcal{S}, n\in \mathbb{N}\}$ (see \cite{Roy}).  Then every member of $\mathcal{A}_0$ is an $F_\sigma$ set.  We restrict the definition of $\mu_D$ to $\mathcal{A}_0$.

\begin{customcla}{3}
If $A, A_1,\dots, A_m,\dots,\in \mathcal{A}_0$ with $A=\bigsqcup_{m=1}^\infty A_m$ then one has $\mu_D(A)=\sum_{m=1}^\infty\mu_D(A_m)$.
\end{customcla}
\begin{proof}
If there is one $A_m$ such that $\mu_D(A_m)=\infty$, the equality holds trivially. Therefore we may assume that each $\mu_D(A_m)$ is finite.  Since each $A_m$ is an $F_\sigma$ set, we can write $A_m= \bigcup_{n=1}^\infty F_{m,n}$ for an increasing sequence of closed sets $\{F_{m,n}: n\in \mathbb{N}\}$. Thus $A=\bigsqcup_{m=1}^\infty \bigcup_{n=1}^\infty F_{m,n}$. Fix an $\epsilon>0$. By Claim 2 for each $m\in \mathbb{N}$ we can choose $N_m$ big enough such that
\[|\mu_D(A_m)-\mu_D(F_{m, N_m})|\leq \epsilon/2^m .\]
In addition, we can make the sequence $\{N_m\}$ strictly increasing. Now Define $P_M=\bigsqcup_{m=1}^M \bigcup_{n=1}^{N_M} F_{m,n}=\bigsqcup_{m=1}^M F_{m, N_M}$ for $M>0$. Note that $\{P_M: M\in \mathbb{N}\}$ is an increasing sequence of closed sets such that $A=\bigcup_{M=1}^\infty P_M$. Then Claim 2 shows that $\mu_D(A)=\sup_M\{\mu_D(P_M)\}$.

Observe that $\mu_D(P_M)=\sum_{m=1}^M \mu_D(F_{m, N_M})\leq \sum_{m=1}^\infty \mu_D(A_m)$ for each $M$. This implies that $\mu_D(A)\leq \sum_{m=1}^\infty \mu_D(A_m)$. Now if $\mu_D(A)=\infty$ then equality holds trivially. So we consider the case that $\mu_D(A)<\infty$. In this case, for every $M>0$, one has
\[|\mu(P_M)-\sum_{m=1}^M \mu_D(A_m)|\leq \sum_{m=1}^M|\mu_D(F_{m, N_M})-\mu_D(A_m)|\leq \sum_{m=1}^M \epsilon/2^m\leq \epsilon.\]
This implies that
\[\infty>\mu_D(A)\geq \mu_D(P_M)\geq \sum_{m=1}^M \mu_D(A_m)-\epsilon,\]
and thus we have
\[\mu_D(A)\geq  \sum_{m=1}^\infty \mu_D(A_m)\]
since $\epsilon$ was arbitrary.

\end{proof}
Claim 3 shows that $\mu_D$ on $\mathcal{A}_0$ is indeed a premeasure.  In addition, it also shows that $\mu_D$ has subadditivity for countably many sets in $\mathcal{A}_0$, i.e., for $A_1,
\dots\in \mathcal{A}_0$,
\[\mu_D(\bigcup_{n=1}^\infty A_n)\leq \sum_{n=1}^\infty \mu_D(A_n).\]

The definition of $\mu_D$ implies that it is $G$-invariant and satisfies inner regularity for all sets in $\mathcal{A}_0$ and outer regularity for closed sets.  We verify the outer regularity for all sets in $\mathcal{A}_0$. Let $B\in \mathcal{A}_0$, which is a $F_\sigma$ set, say, $B=\bigcup^\infty_{n=1}F_n$ for a increasing sequence of closed sets $\{F_n\}$. If $\mu_D(B)=\infty$ then it satisfies the outer regularity trivially since $\mu_D$ is monotone on $F_\sigma$ sets. Now suppose that $\mu_D(B)<\infty$. Then Claim 2 shows that $\mu_D(B)=\sup_{n\in \mathbb{N}}\mu_D(F_n)<\infty$. Then since we have outer regularity for all closed sets, for $\epsilon>0$ and each $n\in \mathbb{N}$, there is an open set $O_n$ such that $F_n\subset O_n$ and
\[\mu_D(F_n)>\mu_D(O_n)-\epsilon/2^n\]
Then define $O=\bigcup_{n=1}^\infty O_n$. Then one has $B\subset O$ and
\[\mu(O\setminus B)\leq \mu_D(\bigcup_{n=1}^\infty(O_n\setminus F_n))\leq \sum^\infty_{n=1}\mu_D(O_n\setminus F_n)<\epsilon.\]
This shows that $\mu_D(B)=\inf\{\mu_D(O): B\subset O,\ O\ \textrm{open}\}$ and thus $\mu_D$ satisfies the outer regularity and thus belongs to $\operatorname{Pr}_G(X)$.
\end{proof}

Recall that the premeasure $\mu_D$ can be extended to a Borel measure on $X$. The extension is unique if $\mu_D$ is $\sigma$-finite on $\mathcal{A}_0$. This happens, in particular, in the case that $D$ is bounded. i.e., $D([1_X])<\infty$.

\begin{theorem}
The map $S: D\rightarrow \mu_D$ is an affine bijection from $\operatorname{Lsc}(W(X, G))$ to $\operatorname{Pr}_G(X)$. In particular, there is an affine bijection between $\operatorname{Lsc}_1(W(X, G))$ and $M_G(X)$ where $\operatorname{Lsc}_1(W(X, G))$ is the set of all states $D$ in $\operatorname{Lsc}(W(X, G))$ with $D([1_X])=1$.
\end{theorem}
\begin{proof}
By Lemma 3.6, $S: D\rightarrow \mu_D$ is well defined. It is not hard to see that  $S$ is affine.  We first show that $S$ is injective. If $\mu_{D_1}=\mu_{D_2}$ then for every $f\in C(X)_+$ one has $$D_1([f])=\mu_{D_1}(\operatorname{supp}(f))=\mu_{D_2}(\operatorname{supp}(f))=D_2([f]),$$
which shows that $D_1=D_2$. To see that $S$ is surjective, it suffices to observe that $S(D_\mu)=\mu$ for every $\mu\in \operatorname{Pr}_G(X)$.

Now if $D([1_X])=1$ then $\mu_D$ is a probability premeasure on $\mathcal{A}_0$ and extends uniquely to a probability Borel measure on $X$ by the remark above. This establishes the last conclusion.
\end{proof}

\begin{lemma}
Suppose that $\alpha: G\curvearrowright X$ is a continuous action of a countably infinite discrete group $G$ on a compact metrizable space $X$. Let $a=(f_1, \dots, f_k)$ and $b=(g_1,\dots, g_l)$ be elements in $K(X, G)$. For all $i\leq k$ and $j\leq l$, we write $A_i=\operatorname{supp}(f_i)$ and $B_j=\operatorname{supp}(g_j)$.  Consider the following statements.
\begin{enumerate}[label=(\roman*)]
	\item There is an $n\in \mathbb{N}^+$ such that $(n+1)[a]\leq n[b]$.
	
	\item there is an $m\in \mathbb{N}^+$ such that $[a]\leq m[b]$ and $D([a])<D([b])$ for every state $D\in S(W(X, G))$ with $D[b]=1$.
	
	\item $\bigcup_{i=1}^kA_i\subset G\cdot (\bigcup_{j=1}^l B_j)$ and $\sum_{i=1}^k\mu(A_i)<\sum_{j=1}^l\mu(B_j)$ for every $\mu\in \operatorname{Pr}_G(X)$ with  $\sum_{j=1}^l\mu(B_j)=1$.
\end{enumerate}
Then (i)$\Rightarrow$(ii)$\Rightarrow$(iii).
\end{lemma}
\begin{proof}
	In light of Proposition 3.1, it suffices to prove 	(ii)$\Rightarrow$(iii).
	
	(ii)$\Rightarrow$(iii).  Let $x\in A_i$ for some $i\leq k$. Now suppose $[a]\leq m[b]$ for some $m\in \mathbb{N}^+$. Then one has 
	\[(f_1,\dots, f_k)\preccurlyeq (g_1,\dots, g_l, g_1,\dots, g_l,\dots, g_1,\dots, g_l);\]
	where $ (g_1,\dots, g_l, g_1,\dots, g_l,\dots, g_1,\dots, g_l)$ is the concatenation of $(g_1,\dots, g_l)$ with itself for $m$ times. This implies that there is an $s\in G$ and an open set $O$ containing $x$ such that $sO\subset B_j$ for some $j\leq l$. This shows $\bigcup_{i=1}^kA_i\subset G\cdot (\bigcup_{j=1}^l B_j)$.  In addition,  for each $\mu\in \operatorname{Pr}_G(X)$ with $\sum_{j=1}^l\mu(B_j)=1$, one has 
	\[D_\mu([b])=\sum_{j=1}^l\mu(B_j)=1,\] where $D_\mu$ is the lower semi-continuous state induced by $\mu$. Thus, (ii) implies that $D_\mu([a])<D_\mu([b])$, which means
	 \[\sum_{i=1}^k\mu(A_i)<\sum_{j=1}^l\mu(B_j),\]
	 as desired.
	\end{proof}

Now we are able to prove the following result. 

\begin{theorem}
Suppose that $\alpha: G\curvearrowright X$ is a continuous action of a countably infinite discrete group $G$ on a compact metrizable space $X$.  Let $k, l\in \mathbb{N}$. Then the following are equivalent.
\begin{enumerate}[label=(\roman*)]
\item For any $a=(f_1, \dots, f_k)$ and $b=(g_1,\dots, g_l)\in K(X, G)$, if there is an $n\in \mathbb{N}^+$ such that $(n+1)[a]\leq n[b]$ then $[a]\leq [b]$.

\item For any $a=(f_1, \dots, f_k)$ and $b=(g_1,\dots, g_l)\in K(X, G)$ with  $A_i=\operatorname{supp}(f_i)$ and $B_j=\operatorname{supp}(g_j)$ for all $i\leq k$ and $j\leq l$, if $\bigcup_{i=1}^kA_i\subset G\cdot (\bigcup_{j=1}^l B_j)$ and $\sum_{i=1}^k\mu(A_i)<\sum_{j=1}^l\mu(B_j)$ for every $\mu\in \operatorname{Pr}_G(X)$ with  $\sum_{j=1}^l\mu(B_j)=1$, then $a\preccurlyeq b$.

\end{enumerate}
\end{theorem}
\begin{proof}
	
	(i)$\Rightarrow$(ii) Suppose (i) holds and $a=(f_1, \dots, f_k)$ and $b=(g_1,\dots, g_l)\in K(X, G)$ satisfy the assumptions in (ii). Our aim is to establish $a\preccurlyeq b$. We begin with  $\bigcup_{i=1}^kA_i\subset G\cdot (\bigcup_{j=1}^l B_j)$. Let $\epsilon>0$ and define $F_i=\overline{\operatorname{supp}((f_i-\epsilon)_+)}$ for each $i\leq k$. Define $F=\bigcup_{i=1}^kF_i$. Then one has
	\[F\subset \bigcup_{i=1}^kA_i\subset G\cdot (\bigcup_{j=1}^l B_j),\] which implies that there is a $d\in \mathbb{N}^+$ and $s_1,\dots, s_d\in G$ such that for any $i\leq k$, one has
	\[\operatorname{supp}((f_i-\epsilon)_+)\subset F_i\subset F\subset \bigcup_{p=1}^d\bigcup_{j=1}^ls_pB_j.\]
	This entails that $[(f_i-\epsilon)_+]\leq d[(g_1,\dots, g_l)]$ for each $i\leq k$ and thus one has
	\[[(a-\epsilon_+)]=[((f_1-\epsilon)_+,\dots, (f_k-\epsilon)_+)]\leq kd[(g_1,\dots, g_l)]=kd[b].\]
	Now set $m=kd$ and thus one has   $[(a-\epsilon)_+]\leq m[b]$.
	
	On the other hand, we  claim that for all $\mu\in \operatorname{Pr}_G(X)$ with $0<\sum_{j=1}^l\mu(B_j)\leq 1$ one still has $\sum_{i=1}^k\mu(A_i)<\sum_{j=1}^l\mu(B_j)$. Indeed, for such a premeasure $\mu$, define 
	\[\mu'(\cdot)=\mu(\cdot)/(\sum_{j=1}^l\mu(B_j)),\]
	which is a premeasure in $\operatorname{Pr}_G(X)$ with $\sum_{j=1}^l\mu'(B_j)=1$. Then one has
	\[\sum_{i=1}^k\mu'(A_i)<\sum_{j=1}^l\mu'(B_j)=1\] 
	by the second assumption in (ii). This shows that
	 \[\sum_{i=1}^k\mu(A_i)<\sum_{j=1}^l\mu(B_j)\]. 
	 
	 Then, this claim  implies that $D'([a])<D'([b])$ for all $D'\in \operatorname{Lsc}(W(X, G))$ with $0<D'([b])\leq 1$ by Theorem 3.7. Therefore, for any state $D\in S(W(X, G))$ with $D([b])=1$, since $\overline{D}$ is always lower semi-continuous by Proposition 3.4, we have 
	 \[D[(a-\epsilon)_+]\leq \overline{D}([a])<\overline{D}([b])\leq D([b])=1.\]
	 Then Proposition 3.1 implies that there is an $n\in \mathbb{N}^+$ such that $(n+1)[(a-\epsilon)_+]\leq n[b].$ Then (i) implies that $[(a-\epsilon)_+]\leq [b]$, i.e., $(a-\epsilon)_+\preccurlyeq b$. Since $\epsilon$ is arbitrary, one has $a\preccurlyeq b$ as desired.
	 
	 (ii)$\Rightarrow$(i). Let $a=(f_1, \dots, f_k)$ and $b=(g_1,\dots, g_l)\in K(X, G)$ such that  $(n+1)[a]\leq n[b]$ holds for some $n\in \mathbb{N}^+$.  Then (i)$\Rightarrow$(iii) in Lemma 3.8  shows that the assumption in (ii) holds, Since we have assumed that (ii) holds. one has $a\preccurlyeq b$, i.e., $[a]\leq [b]$. This establishes (i).
\end{proof}

We then have the following corollaries.

\begin{corollary}
Suppose that $\alpha: G\curvearrowright X$ is a continuous action of a countably infinite discrete group $G$ on a compact metrizable space $X$. Then the following are equivalent.
\begin{enumerate}[label=(\roman*)]
	\item $W(X, G)$ is almost unperforated.
	
	\item For any two sequences $(A_1,\dots, A_k)$ and $(B_1,\dots, B_l)$ of open sets in $X$, if $\bigcup_{i=1}^kA_i\subset G\cdot (\bigcup_{j=1}^l B_j)$ and $\sum_{i=1}^k\mu(A_i)<\sum_{j=1}^l\mu(B_j)$ for every $\mu\in \operatorname{Pr}_G(X)$ with  $\sum_{j=1}^l\mu(B_j)=1$, then 
	\[\bigsqcup_{i=1}^k A_i\times \{i\}\prec \bigsqcup_{j=1}^l B_j\times \{j\}\] holds.
\end{enumerate}
\end{corollary}
\begin{proof}
This is a direct application of Theorem 3.9 and Definition 2.1.
\end{proof}

\begin{corollary}
	Suppose that $\alpha: G\curvearrowright X$ is a continuous action of a countably infinite discrete group $G$ on a compact metrizable space $X$. Then the following are equivalent. 
	\begin{enumerate}[label=(\roman*)]
		\item For any non-empty open sets $A, B$, if there is an $n\in \mathbb{N}^+$ such that  $(n+1)A\prec nB$ then $A\prec B$.
		
		\item The action $\alpha: G\curvearrowright X$ has (generalized) dynamical comparison in the sense of Definition 1.3.
	\end{enumerate}
\end{corollary}
\begin{proof}
	This is also a direct application of Theorem 3.9 to any pair of nonempty open sets $A, B$ with two functions $f, g\in C(X)_+\subset K(X, G)$ such that $A=\operatorname{supp}(f)$ and $B=\operatorname{supp}(g)$.
	\end{proof}

\begin{corollary}
Suppose that $\alpha: G\curvearrowright X$ is a continuous action of a countably infinite discrete group $G$ on a compact metrizable space $X$. Suppose in addition that $G$ is amenable or $\alpha$ is minimal.   The following are equivalent.
\begin{enumerate}[label=(\roman*)]
 \item Whenever $A, B$ are open sets in $X$ such that $\mu(B)>0$ for all $\mu\in M_G(X)$, if there is an $n\in \mathbb{N}^+$ such that $(n+1)A\prec nB$, then $A\prec B$.

 \item The action $\alpha: G\curvearrowright X$ has dynamical comparison in the sense of Definition 1.2.
\end{enumerate}
\end{corollary}
\begin{proof}
(i)$\Rightarrow$(ii). Let $A, B$ be open sets in $X$. Suppose that $\nu(A)<\nu(B)$ for every $\nu\in M_G(X)$. First this implies $\nu(B)>0$ for all $\nu\in M_G(X)$ and, in particular, $B$ is not empty. When $\alpha$ is minimal or $G$ is amenable, we claim that $X=G\cdot B$. In the case that $\alpha$ is minimal, one has $X=G\cdot B$ trivially. Suppose that $G$ is amenable and $X\neq G\cdot B$, there is a $G$-invariant probability measure $\lambda$ for the closed subsystem $C=X\setminus G\cdot B\neq \emptyset$ since $G$ is amenable. However $\lambda$ induces a probability measure $\lambda'$ on $X$ with $\lambda'(E)=\lambda(E\cap C)/\lambda(C)$ for every Borel set $E$. Observe that $\lambda'(B)=0$ and this is a contradiction. This establishes the claim.  Therefore,  one has $A\subset X=G\cdot B$. In addition, since $X$ is actually covered by finitely many translates of $B$, for every $\mu\in \operatorname{Pr}_G(X)$ with $\mu(B)=1$, one has $\mu(X)$ is finite.  Define $\nu'(\cdot)=\mu(\cdot)/\mu(X)$, which is a probability premeasure in $\operatorname{Pr}_G(X)$. Now extend $\nu'$ to obtain a probability measure in $M_G(X)$, which we still denote by $\nu'$.  Then since $\nu'(A)<\nu'(B)$ holds by assumption, one has $\mu(A)<\mu(B)=1$.  Since (i) is assumed to hold, the same proof of Theorem 3.9(i)$\Rightarrow$(ii) implies $A\prec B$. This establishes (ii).

(ii)$\Rightarrow$(i). Let $A, B$ be open sets in $X$ such that $\mu(B)>0$ for all $\mu\in M_G(X)$ and there is an $n\in \mathbb{N}^+$ such that $(n+1)A\prec nB$. Then one has
\[(n+1)\mu(A)\leq n\mu(B)\] for all $\mu\in M_G(X)$. This implies that $\mu(A)<\mu(B)$ for all $\mu\in M_G(X)$.
Now since $\alpha$ has dynamical comparison, we have $A\prec B$.

\end{proof}

\section{Acknowledgement}
The author should like to thank his supervisor David Kerr for communicating the topic of this paper to him as well as his helpful corrections and discussions. He also should like to thank Christopher Schafhauser for invaluable comments and suggestions which improved this paper a lot.  He would like to thank Jianchao Wu for communicating with him that, restricting to zero-dimensional spaces, the generalized type semigroup introduced in this paper is equivalent to the so-called dynamical Cuntz semigroup in ongoing work of Jianchao Wu and his collaborators. Finally, he should like to thank the anonymous referee whose comments and suggestions helped a lot to improve the paper.


\begin{thebibliography}{10}
\bibitem{A-P-T}P. Ara, F. Perera and A. S. Toms. \textit{K-theory for operator algebras. Classification of $C^\ast$-algebras}, page 1-71 in: \textit{Aspects of Operator Algebras and Applications}. P. Ara, F. Lled\'{o}, and F. Perera (eds.). Contemporary Mathematics vol.534, Amer. Math. Soc., Providence RI, 2011.

\bibitem{B-H}B. Blackadar and D. Handelman. Dimension functions and traces on $C^*$-algebras.  \textit{J. Funct. Anal.}, \textbf{45}(1982), 297-340.


\bibitem{B-L}C. B\"{o}nicke and K. Li. Ideal structure and pure infiniteness of ample groupoid $C^\ast$-algebras. \textit{Ergod. Th. and Dynam. Sys.} \textbf{40} (2020), 34-63.

\bibitem{B}J. Buck. Smallness and comparison properties for minimal dynamical systems. Preprint, 2013, arXiv:1306.6681v1.

\bibitem{C-J-K-M-S-T}C. Conley, S. Jackson, D. Kerr, A. Marks, B. Seward, and R. Tucker-Drob. F{\o}lner tilings for actions of amenable groups.  \textit{Math. Annalen} (2)\textbf{371} (2018), 663-683.

\bibitem{Dow-Z}T. Downarowicz and G. Zhang. The comparison property of amenable groups. arXiv: 1712.05129

\bibitem{Folland}G. B. Folland. \textit{Real Analysis: Modern Techniques and Their Applications (Pure and Applied Mathematics (New York))}. Wiley Interscience, New York, 1984.

\bibitem{G-W}E. Glasner and B. Weiss. Weak orbit equivalence of Cantor minimal systems. \textit{Internat. J. Math.} \textbf{6} (1995), 559-579.

\bibitem{D}D. Kerr. Dimension, comparison, and almost finiteness. To appear in \textit{J. Eur. Math. Soc.}

\bibitem{D-G}D. Kerr and G. Szab\'{o}, Almost finiteness and the small boundary property. To appear in \textit{Comm. Math. Phys.} DOI: 10.1007/s00220-019-03519-z





\bibitem{L-S}M. Laca and J. Spielberg. Purely infinite $C^\ast$-algebras from boundary actions of discrete groups. \textit{J. Reine. Angew. Math.} \textbf{480}(1996), 125-139.

\bibitem{M1}X. Ma. Invariant ergodic  measures and the classification of crossed product $C^\ast$-algebras.  \textit{J. Funct. Anal.} \textbf{276}(2019), 1276-1293

\bibitem{M2}X. Ma. Comparison and pure infiniteness of crossed products. \textit{Trans. Amer. Math. Soc.} \textbf{372}(2019), no. 10,  7497-7520

\bibitem{J-R}P. Jolissaint and G. Robertson. Simple purely infinite $C^\ast$-algebras and $n$-filling actions. \textit{J. Funct. Anal.} \textbf{175}(2000), 197-213.

\bibitem{O-P-R}E. Ortega, F. Perera, and M. R{\o}rdam. The corona factorization property, stability, and the Cuntz semigroup of a $C^\ast$-algebra. \textit{Int. Math. Res. Not.} \textbf{2012},  34-66.

\bibitem{TimR}T. Rainone. Finiteness and paradoxical decompostions in $C^\ast$-dynamical systems.  \textit{J. Noncommut. Geom.}, 11 (2017), no. 2, 791-822.


\bibitem{R}M. R{\o}rdam. The stable and the real rank of $\mathcal{Z}$-absorbing $C^\ast$-algebras. \textit{Internat. J. Math.} \textbf{15} (2004), 1065-1084.


\bibitem{R-S}M. R{\o}rdam and A. Sierakowski. Purely infinite $C^\ast$-algebras arising from crossed products. \textit{ Ergod. Th. and Dynam. Sys.} \textbf{32} (2012), 273-293.

\bibitem{Roy}H. L. Royden. \textit{Real Analysis} (3rd ed.), Macmillan, New York (1988).


\bibitem{Wagon}S. Wagon. \textit{The Banach-Tarski Paradox}. Cambridge University Press, Cambridge, 1993.

\end{thebibliography}
\end{document}